\documentclass[12pt,reqno]{amsart}
\usepackage[top=2.5cm, left=3cm, right=3cm, bottom=2.5cm]{geometry}
\usepackage{amscd,amsmath,latexsym,amssymb,mathtools,amsthm}
\usepackage{thmtools}
\usepackage{enumerate} 

\usepackage[dvipsnames]{xcolor}
\usepackage{graphicx}
\usepackage{tikz}
\usepackage{wrapfig}
\usepackage{float}
\usepackage{caption}
\usepackage{subcaption}

\usepackage{tikz-cd}
\usepackage{tkz-euclide}
\usepackage{multicol}

\definecolor{lnkcol}{rgb}{0,0.25,0.75}
\usepackage[colorlinks=true, linkcolor=lnkcol, citecolor=magenta, urlcolor=blue]{hyperref}

\newtheorem{introthm}{Theorem}

\newtheorem{introcor}[introthm]{Corollary}

\newtheorem{introconj}[introthm]{Conjecture}

\newtheorem{theorem}{Theorem}[section]

\newtheorem{proposition}[theorem]{Proposition}
\newtheorem{lemma}[theorem]{Lemma}
\newtheorem{definition}[theorem]{Definition}

\numberwithin{figure}{section}
\numberwithin{equation}{theorem}

\title{On 2-Sphere Bowditch Boundaries Attaining Conformal Dimension} 

\author{Abhijit Pal}
\address{Indian Institute of Technology Kanpur}
\email{abhipal@iitk.ac.in}

\author{Rana Sardar}
\address{Indian Institute of Science Education and Research, Mohali}
\email{ranasardar@iisermohali.ac.in}

\date{\today}

\begin{document}

\begin{abstract} 
    Bonk and Kleiner proved that if $G$ is a Gromov hyperbolic group whose boundary $\partial_{\infty}G$ is homeomorphic to an Ahlfors $Q$-regular metric $2$-sphere $Z$, and the Ahlfors regular conformal dimension of $Z$ is attained and equal to $Q$, then $G$ acts discretely, cocompactly, and isometrically on $\mathbb{H}^3$. In this article, we extend the Bonk-Kleiner theorem to the setting of relatively hyperbolic groups. More precisely, we prove that if $(G,\mathcal{H})$ is a relatively hyperbolic group whose Bowditch boundary is homeomorphic to an Ahlfors $Q$-regular metric $2$-sphere $Z$, with the Ahlfors regular conformal dimension of $Z$ attained and equal to $Q$, then $G$ acts discretely and isometrically on $\mathbb{H}^3$, and every subgroup in $\mathcal{H}$ is virtually $\mathbb Z^2$.
\end{abstract}

\maketitle

\section{Introduction}

Given a finitely generated group $G$, a fundamental problem in geometric group theory is to determine when $G$ is Kleinian. Recall that a \emph{Kleinian group} is a discrete subgroup of \(\operatorname{PSL}(2,\mathbb{C}) \cong \operatorname{Isom}^+(\mathbb{H}^3)\), the group of orientation-preserving isometries of hyperbolic $3$-space $\mathbb{H}^3$. We say that $G$ is a \emph{cocompact Kleinian group} if its action on $\mathbb{H}^3$ is cocompact. One of the central conjectures in this direction is Cannon's  Conjecture, which predicts that a hyperbolic group with $2$-sphere boundary is a cocompact Kleinian Group.

\begin{introconj}[Cannon's  Conjecture, {\cite[Conjecture~9.1]{KB2002}, \cite[Problem~53]{Kapovich2005}}]
    Let $G$ be a hyperbolic group whose Gromov boundary $\partial_\infty G$ is homeomorphic to the standard $2$-sphere $\mathbb{S}^2$. Then $G$ admits a discrete, cocompact, isometric action on hyperbolic $3$-space $\mathbb{H}^3$
\end{introconj}

A natural extension of this problem arises for relatively hyperbolic groups. Let $G$ be hyperbolic relative to 
 An analogue of Cannon's  Conjecture for relatively hyperbolic groups was proposed in \cite[Problem~57]{Kapovich2005}. 

\begin{introconj}[Relative Cannon's  Conjecture] 
    Let $G$ be a group hyperbolic relative to a collection ${H_1,\ldots,H_k}$ of virtually $\mathbb{Z}^2$ subgroups. If the Bowditch boundary of $G$ is homeomorphic to $\mathbb{S}^2$, then $G$ is commensurable with the fundamental group of a finite-volume hyperbolic $3$-manifold.
\end{introconj}

 Bonk and Kleiner proved that a compact Ahlfors $Q$-regular Loewner space homeomorphic to the standard sphere $\mathbb{S}^2$ must have $Q=2$ and must be quasisymmetrically equivalent to $\mathbb{S}^2$; see \cite[Theorem~1.2]{BK2002Qs}. Combining this result with their work on boundaries of hyperbolic groups, they obtained the following consequence: if the boundary of a hyperbolic group attains its Ahlfors regular conformal dimension, then the group is a cocompact Kleinian group; see \cite{BK2005}.

The purpose of this paper is to establish a corresponding Loewner criterion for compact metric spaces equipped with certain quasisymmetric group actions. 

\begin{introthm}\label{thm_main_1}
    Let $(Z,d_Z)$ be a compact Ahlfors $Q$-regular metric space whose Ahlfors regular conformal dimension is $Q>1$. Suppose that $Z$ admits a fixed point free, uniformly quasisymmetric action $G\curvearrowright Z$ such that the induced action on the space of distinct pairs $\operatorname{Pair}(Z)$ is cocompact. Then $Z$ is $Q$-Loewner.
\end{introthm}

This theorem has an immediate consequence if in additional the space $Z$ is homeomorphic to the \(2\)-sphere. By the Bonk-Kleiner's theorem \cite[Theorem~1.2]{BK2002Qs}  \(Q=2\) and \(Z\) is quasisymmetrically equivalent to the standard \(2\)-sphere.

\begin{introcor}\label{thm_main_2}
    Let $(Z,d_Z)$ be an Ahlfors $Q$-regular metric space, homeomorphic to the standard sphere $\mathbb{S}^2$, and suppose that its Ahlfors regular conformal dimension is attained and equal to $Q\geq 2$. Suppose that $Z$ admits a fixed point free, uniformly quasisymmetric action $G\curvearrowright Z$ such that the induced action on the space of distinct pairs $\operatorname{Pair}(Z)$ is cocompact. Then, $Q = 2$ and $Z$ is quasisymmetric to $\mathbb{S}^2$.
\end{introcor}

We now apply the preceding results to obtain a `relative Cannon conjecture' type result for relatively hyperbolic groups whose Bowditch boundary is a $2$-sphere. Let $G$ be a group hyperbolic relative to a finite collection $\mathcal H$ of subgroups.
We denote by $Z=\partial(G,\mathcal{H})$
the Bowditch boundary of the relatively hyperbolic pair $(G,\mathcal H)$. Then $Z$ is compact and $G$ admits a fixed-point-free, uniformly quasisymmetric (equivalently, quasi-Möbius) action on $Z$; see \cite{Mackay-Sisto,PS2026,Sardar2026}. Moreover, the induced action $G\curvearrowright \operatorname{Pair}(Z)$ on the space of distinct pairs is cocompact. 
Now assume, in addition, that $Z$ is homeomorphic to the standard sphere $\mathbb{S}^2$ and is Ahlfors $Q$-regular, where $Q\geq 2$. Since every Ahlfors regular metric space is uniformly perfect and doubling, it follows that $Z$ has these properties. Further, assume that $Q$ is the Ahlfors regular conformal dimension of $Z$. Thus, the Ahlfors regular conformal dimension of $Z$ is attained. Applying Corollary~\ref{thm_main_2}, there exists a quasisymmetric homeomorphism 
\[
h : Z \to \mathbb{S}^2.
\] 
Conjugating the action $G \curvearrowright Z$ by $h$ gives an action $G \curvearrowright \mathbb{S}^2$ which is still uniformly quasisymmetric. 
Since every quasisymmetric self-homeomorphism of $\mathbb S^2$ is quasiconformal, this gives a uniformly quasiconformal action on $\mathbb S^2$.
A theorem of Sullivan and Tukia (see \cite[p.~468]{Sullivan1981} and \cite[Theorem~F and Remark~F2]{Tukia1986}) states that a uniformly quasiconformal group action on $\mathbb{S}^2$ is quasiconformally conjugate to an action by M\"{o}bius transformations. 
Thus, once the boundary has been identified quasisymmetrically with $\mathbb S^2$, the dynamics of the original action can be straightened to a conformal action. Hence, $G$ is a discrete subgroup of $\operatorname{PSL}(2,\mathbb{C})$. 
Mackay-Sisto {\cite[Proposition~4.5]{MS2020}} showed that Bowditch boundary is doubling if and only if every  subgroup in $\mathcal H$ is virtually nilpotent. The nilpotent subgroups of $\operatorname{PSL}(2,\mathbb{C})$ are virtually abelian. The subgroups in $\mathcal{H}$ are virtually abelian.
A subgroup $H\in\mathcal H$ is discrete and acts cocompactly on a horosphere. So, rank of $H$ is $2$ and $H$ is virtually $\mathbb Z^2$. Thus, we obtain the following relative Cannon conjecture-type result.

\begin{introthm}\label{thm_our_rel_conjec}
    Let $(G,\mathcal{H})$ be a relatively hyperbolic group whose Bowditch boundary $\partial(G,\mathcal{H})$ is homeomorphic to the standard $2$-sphere $\mathbb{S}^2$ and is Ahlfors regular with respect to a visual metric. If the Ahlfors regular conformal dimension of $\partial(G,\mathcal{H})$ is attained, then $G$ is a Kleinian group. Moreover, the peripheral subgroups are virtually $\mathbb Z^2$. 
\end{introthm}

It is also natural to ask whether the condition in Theorem~\ref{thm_our_rel_conjec} is necessary for Kleinian groups. In the setting relevant to the Relative Cannon's Conjecture, the answer is affirmative at the level of the boundary conformal structure. Indeed, let $G$ be a Kleinian group and let $\mathcal{H}$ be a finite collection of finitely generated, virtually abelian maximal parabolic subgroups such that $(G,\mathcal H)$ is relatively hyperbolic. Then $G$ acts properly discontinuously on $\mathbb H^3$ and cocompactly on the truncated hyperbolic space $\mathbb{H}^3 \setminus \mathcal{H}^h$, where $\mathcal H^h$ is the corresponding collection of horoballs. Hence, by the Milnor-Švarc lemma, $G$ is quasi-isometric to $\mathbb{H}^3 \setminus \mathcal{H}^h$. Results of Mackay-Sisto and Pal-Sardar imply that the Bowditch boundary is quasisymmetrically, equivalently quasi-M\"{o}bius, homeomorphic to $\mathbb S^2$; see \cite[Corollary~1.3]{Mackay-Sisto} and \cite[Theorem~1]{PS2026}, see also \cite{Sardar2026}. Since the standard sphere is Ahlfors $2$-regular and $\operatorname{Cdim_{AR}}(\mathbb{S}^2)=2$, it follows that $\operatorname{Cdim_{AR}}\bigl(\partial(G,\mathcal H)\bigr)=2,$ and this conformal dimension is attained.

There exists a relatively hyperbolic group $(G,\mathcal{H})$ whose Bowditch boundary is homeomorphic, but not quasisymmetric, to $\mathbb S^2$. An example is obtained as follows. Let $M$ be a compact hyperbolic $3$-manifold with nonempty totally geodesic boundary
$\partial M=S_g$, $g\ge2$, and set
\[
G=\pi_1(M),\qquad H=\pi_1(S_g).
\]
Then $H$ is almost malnormal and quasiconvex in $G$, so $(G,H)$ is relatively hyperbolic by Bowditch's theorem \cite[Theorem~7.11]{Bowditch2012}. The limit set of $G$ is a Sierpi\'nski carpet, and collapsing its boundary circles to points gives $\partial(G,H)\cong\mathbb S^2$ by Moore's Decomposition Theorem \cite{Moore1925} and Manning's description of the Bowditch boundary \cite{Manning2020}. However, the Bowditch boundary $\partial(G,H)$ is not Ahlfors regular, since the parabolic subgroup $H$ is not virtually nilpotent, by \cite[Proposition~4.1]{MS2020}. Consequently, $\partial(G,H)$ is not quasisymmetric to $\mathbb S^2$.

\bigskip

Our approach is inspired by the work of Bonk and Kleiner~\cite{BK2005}. We extend their argument by replacing cocompactness on triples with the weaker assumption of cocompactness on pairs. The proof relies on the equivalence between quasi-Möbius and quasisymmetric maps for bounded metric spaces.

\subsection*{Outline of the paper}
The paper is organized as follows. Section~2 recalls the necessary background. In Section~3, we study cocompact actions on distinct pairs and their dynamical properties. Section~4 introduces weak tangents and studies their behaviour under uniformly quasisymmetric actions. In Section~5, we develop the thick-path machinery and establish the existence and density of thick paths. Finally, Section~6 combines these results with modulus estimates to prove the Loewner property and Theorem~\ref{thm_main_1}.

\newpage
\section{Preliminaries}
\label{sec:preliminaries} 

\subsection{Relatively Hyperbolic Groups and Boundaries}
\leavevmode 
\smallskip

Several notions of relative hyperbolicity appear in the literature. We refer to Hruska’s article \cite{Hruska2010} for a discussion of the equivalence of these definitions.
In this article,  we adopt the dynamical characterization of relative hyperbolicity due to Gerasimov \cite[Definition $RH_{32}$]{Gera2012} as our definition. 

\begin{definition}[Relatively hyperbolic group] \label{defn:rel_hyp_gp}
    Let $G$ be a finitely generated discrete group, and let $\mathcal{H}=\{H_1,\ldots,H_m\}$ be a finite collection of subgroups of $G$. We say that the pair $(G,\mathcal{H})$ is \emph{relatively hyperbolic} if there exists a compact Hausdorff space $Z$ together with an action \(G\curvearrowright Z\) by homeomorphisms satisfying the following conditions:
    \begin{enumerate}[$(i)$]
        \item The action is a \emph{$3$-proper} (or \emph{convergence action}); that is, the induced action on the space of distinct ordered triples $\operatorname{Tri}(Z)$ is proper, where
        \[
        \operatorname{Tri}(Z)=\{(x,y,z)\in Z^3 : x,y,z \text{ are pairwise distinct}\}.
        \]

        \item The action is \emph{$2$-cocompact} (or \emph{expansive}); namely, \(\operatorname{Pair}(Z)/G\) is compact, where
    \[
    \operatorname{Pair}(Z)=\{(x,y)\in Z\times Z : x\neq y\}.
    \]

        \item The collection $\mathcal{H}$ consists of representatives of the conjugacy classes of stabilizers of the parabolic points of the action.
    \end{enumerate}

    Here, a point $p\in Z$ is called \emph{parabolic} if $\operatorname{Stab}_G(p)$ is infinite and acts cocompactly on $Z\setminus\{p\}$. The subgroup $\operatorname{Stab}_G(p)$ is called the \emph{parabolic subgroup} associated to $p$.
\end{definition} 

The compactum $Z$ in Definition~\ref{defn:rel_hyp_gp} is called  the Bowditch boundary of $(G,\mathcal H)$. We denote it by $\partial(G,\mathcal H)$. In Definition~\ref{defn:rel_hyp_gp}, the induced action $G\curvearrowright\operatorname{Tri}(Z)$ is cocompact if and only if the group $G$ is hyperbolic. 
Moreover, since the natural projection $\operatorname{Tri}(Z)\to \operatorname{Pair}(Z)$ defined by $(x,y,z)\mapsto (x,y)$, is continuous and surjective, every $3$-cocompact action is also $2$-cocompact.

\subsection{Metric Properties and Dimensions}
\label{subsec:metric_properties}

\begin{definition}[Uniformly perfect metric space] 
    A metric space $(Z,d_Z)$ is \emph{uniformly perfect} if there exists $C>1$ such that, for every $a\in Z$ and every $0<R\leq\operatorname{diam}(Z)$, we have $B(a,R)\setminus B\left(a,\frac{R}{C}\right)\neq\varnothing$. 
\end{definition} 

\begin{definition}[Doubling metric space] 
    A metric space $(Z,d_Z)$ is said to be \emph{$N$-doubling} if every ball of radius $R$ can be covered by at most $N$ balls of radius $R/2$. 
\end{definition} 

The boundary of a hyperbolic group is doubling. In contrast, the Bowditch boundary of a relatively hyperbolic group need not be doubling. The following result characterizes doubling of the Bowditch boundary in terms of the peripheral subgroups.

\begin{proposition}[{\cite[Proposition~4.5]{MS2020}}] 
    Let $(G,\mathcal H)$ be relatively hyperbolic. Then its Bowditch boundary is doubling if and only if every peripheral subgroup in $\mathcal H$ is virtually nilpotent. 
\end{proposition}

\begin{definition}[Ahlfors regular metric measure space] \label{defn:Ahlfors_reg} 
    Let $(Z,d,\mu)$ be a metric measure space. For $Q>0$, we say that $(Z,d,\mu)$ is \emph{Ahlfors $Q$-regular} if there exists a constant $C\geq1$ such that for every $z\in Z$ and every $0<r\leq\operatorname{diam}(Z)$, we have
    \[ 
    C^{-1}r^Q \leq \mu(B(z,r)) \leq Cr^Q.
    \]     
\end{definition} 

By a result of Coornaert \cite{Coornaert1993}, the Gromov boundary of a hyperbolic group is Ahlfors regular with respect to a visual metric. In contrast, for a relatively hyperbolic group, a visual metric on the Bowditch boundary need not be Ahlfors regular; \cite[Proposition~1.6]{BM1996}, \cite[Example~1, p.~189]{GP2001}. 

\begin{definition}[Ahlfors regular conformal dimension] \label{defn:conf_dim}
    Let $(Z,d_Z)$ be a metric space. The Ahlfors regular conformal dimension of $Z$ is the infimal Hausdorff dimension of all Ahlfors regular metric spaces quasisymmetrically homeomorphic to $Z$.
    
\end{definition}

 
\subsection{Quasi-M\"{o}bius and Quasisymmetric Maps}
\leavevmode

\begin{definition}
    Let $(Z,d)$ be a metric space. The \emph{cross-ratio} \([z_1,z_2,z_3,z_4]\) of four distinct points $z_1,z_2,z_3,z_4$ in $Z$ is defined by
    \[
    [z_1,z_2,z_3,z_4]
    :=
    \frac{d(z_1,z_3)\, d(z_2,z_4)}
     {d(z_1,z_4)\, d(z_2,z_3)}.
    \]
    
    Let $\eta:[0,\infty)\to[0,\infty)$ be a homeomorphism, and let \(f:(X,d_X)\to(Y,d_Y)\) be an injective map between metric spaces. 
    The map $f$ is called \emph{$\eta$-quasi-M\"{o}bius} if for every four-tuple $(x_1,x_2,x_3,x_4)$ of distinct points in $X$,
    \[
    [f(x_1),f(x_2),f(x_3),f(x_4)]
    \leq
    \eta\,\bigl([x_1,x_2,x_3,x_4]\bigr).
    \]
    The map $f$ is called \emph{$\eta$-quasisymmetric} if for every triple $(x_1,x_2,x_3)$ of distinct points in $X$,
    \[
    \frac{d_Y\bigl(f(x_1),f(x_2)\bigr)}
     {d_Y\bigl(f(x_1),f(x_3)\bigr)}
    \leq \eta\!\left( \frac{d_X(x_1,x_2)} {d_X(x_1,x_3)} \right). 
    \]
\end{definition}

Note that Quasi-M\"{o}bius and quasisymmetric maps are homeomorphisms onto their images. The composition of two quasi-M\"{o}bius maps with distortion functions $\eta_1$ and $\eta_2$ is $\eta_2\circ\eta_1$-quasi-M\"{o}bius; similarly, compositions of quasisymmetric maps are quasisymmetric. Moreover, every $\eta$-quasisymmetric map is $\widetilde{\eta}$-quasi-M\"{o}bius, where $\widetilde{\eta}$ depends only on $\eta$. Conversely, every quasi-M\"{o}bius map between bounded metric spaces is quasisymmetric.

 
\subsection{Modulus of Path Families and Loewner Spaces}
\leavevmode

\begin{definition}[Admissible densities and modulus of a path family]
    Let $(Z,d,\mu)$ be a metric measure space, and $\Gamma$ denotes a family of paths in $Z$.
    
    A \emph{density} on $Z$ is a Borel measurable function \(\rho : Z \to [0,\infty].\) A density $\rho$ is said to be \emph{admissible} for $\Gamma$ if \(\int_{\gamma} \rho \, ds \geq 1\) for every rectifiable path $\gamma \in \Gamma$, where the integral is taken with respect to arc length along $\gamma$. 
    
    For $Q \geq 1$, the \emph{$Q$-modulus} of the path family $\Gamma$ is defined by 
    \begin{equation}\label{eq:modulus}
        \operatorname{Mod}_Q(\Gamma)
        :=
        \inf_{\rho}
        \int_Z \rho^{Q}\, d\mu,
    \end{equation}
    where the infimum is taken over all admissible densities $\rho : Z \to [0,\infty]$ for $\Gamma$.
\end{definition}

The modulus is well behaved under quasisymmetric and quasi-M\"{o}bius maps between Ahlfors regular spaces. The following theorem, due to Tyson, establishes the coarse-invariance of the \(Q\)-modulus of path families under quasisymmetric or quasi-M\"{o}bius homeomorphisms between Ahlfors \(Q\)-regular metric measure spaces, which will be used later.

\begin{theorem}[{\cite{Tyson1998,Tyson2001,BK2005}}]\label{thm_Tyson}
    Let $(X,d_X,\mu_X)$ and $(Y,d_Y,\mu_Y)$ be Ahlfors $Q$-regular locally compact metric measure spaces with $Q\geq 1$, and $\eta:(0,\infty) \to (0,\infty)$ be an increasing homeomorphism. Suppose that \(f:X\to Y\) is an $\eta$-quasisymmetric (or quasi-M\"{o}bius) homeomorphism. Then there exists a constant $K\geq 1$, depending only on $X$, $Y$, and $\eta$, such that for every family $\Gamma$ of paths in $X$,
    \[
    \frac{1}{K}\,\operatorname{Mod}_Q(\Gamma)
    \leq
    \operatorname{Mod}_Q(f\circ \Gamma)
    \leq
    K\,\operatorname{Mod}_Q(\Gamma),
    \]
    where \(f\circ \Gamma = \{\,f\circ\gamma : \gamma\in\Gamma\,\}.\)
\end{theorem}

\begin{definition}[Relative distance and Loewner spaces]
    Let $(Z,d,\mu)$ be a metric measure space. For subsets $E,F\subseteq Z$ of positive diameter, their \emph{relative distance} is defined by 
    \[
    \Delta(E,F)
    :=
    \frac{\operatorname{dist}(E,F)}
    {\min\{\operatorname{diam}(E),\operatorname{diam}(F)\}}.
    \]

    We denote by $\Gamma(E,F)$ the family of all paths in $Z$ joining $E$ to $F$. Suppose that $Z$ is connected and let $Q\geq 1$. We say that $(Z,d,\mu)$ is a \emph{$Q$-Loewner space} if there exists a positive decreasing function \(\Psi:(0,\infty)\to(0,\infty)\) such that 
    \begin{equation}\label{eq:Loewner_condition}
        \operatorname{Mod}_Q\bigl(\Gamma(E,F)\bigr)
        \geq
        \Psi\bigl(\Delta(E,F)\bigr),
    \end{equation}
    whenever $E$ and $F$ are disjoint continua in $Z$.
\end{definition}

The following criterion of Bonk and Kleiner asserts that a space that satisfies a Loewner-type condition for pairs of balls also satisfies the Loewner condition for all pairs of continua. 

\begin{proposition}[{\cite[Prpoposition~3.1]{BK2005}}] \label{prop_Loewner_condition_with_ball}
    Let $(Z,d,\mu)$ be a complete metric measure space. Suppose that for every $C>0$, there exist constants $m = m(C) > 0$ and $L = L(C) > 0$ such that, whenever $R > 0$ and $B, B' \subset Z$ are balls of radius $R$ satisfying $\operatorname{dist}(B, B_0) \le C R$, the family $\{ \gamma \in \Gamma(B, B_0) :  \operatorname{length}(\gamma) \le L R \}$ has $Q$-modulus at least $m$. Then $(Z,d,\mu)$ is a $Q$-Loewner space.
\end{proposition}

\section{Cocompact Actions and Their Dynamics}

In this section, we study the consequences of cocompactness of the induced action on the space of distinct pairs. We first characterize cocompactness in terms of a uniform separation property. We then establish some technical lemmas, which will be used in later sections

Let $(Z,d_Z)$ be a compact metric space. We denote by $\operatorname{Pair}(Z)$ the space of distinct ordered pairs in $Z$, that is, $\operatorname{Pair}(Z) = \{(z_1,z_2)\in Z^2:z_1\neq z_2\}$.
An action $G\curvearrowright Z$ induces an action on $\operatorname{Pair}(Z)$ by $g\cdot(z_1,z_2)=(gz_1,gz_2)$. We say that the action $G\curvearrowright\operatorname{Pair}(Z)$ is \emph{cocompact} if there exists a compact set $K\subset\operatorname{Pair}(Z)$ such that
\[
\operatorname{Pair}(Z)=GK=\bigcup_{g\in G}gK.
\]

\begin{lemma}[Uniform separation characterization]
\label{lem:cocompact_uniformly_separated}
    The action $G\curvearrowright\operatorname{Pair}(Z)$ is cocompact if and only if there exists $\delta>0$ such that, for every $(z_1,z_2)\in\operatorname{Pair}(Z)$, there exists $g\in G$ satisfying $d_Z(gz_1,gz_2)\geq\delta$.
\end{lemma}

\begin{proof}
    Suppose first that $G\curvearrowright\operatorname{Pair}(Z)$ is cocompact. Then there exists a compact set $K\subset\operatorname{Pair}(Z)$ such that
    \[
    \operatorname{Pair}(Z)=GK.
    \]
    The function $d_Z: Z\times Z\to\mathbb R$ is continuous and is strictly positive on $K$. Since $K$ is compact, there exists
    \[
    \delta:=\min\{d_Z(x,y):(x,y)\in K\}>0.
    \]
    Given any $(z_1,z_2)\in\operatorname{Pair}(Z)$, the equality $\operatorname{Pair}(Z)=GK$ gives some $g\in G$ such that $(gz_1,gz_2)\in K$. Consequently, $d_Z(gz_1,gz_2)\geq\delta$. Thus every pair can be mapped by an element of $G$ to a $\delta$-separated pair.

    Conversely, suppose that there exists $\delta>0$ such that every $(z_1,z_2)\in\operatorname{Pair}(Z)$ can be mapped by some element of $G$ to a pair whose distance is at least $\delta$. We define 
    \[
    K_\delta = \{(z_1,z_2)\in Z^2:d_Z(z_1,z_2)\geq\delta\}.
    \]
    Since $d_Z$ is continuous, $K_\delta$ is closed in the compact space $Z^2$ and hence is compact. Moreover, since $\delta>0$, $K_\delta\subset\operatorname{Pair}(Z)$. By assumption, for every $(z_1,z_2)\in\operatorname{Pair}(Z)$ there exists $g\in G$ such that $(gz_1,gz_2)\in K_\delta$. Equivalently, $(z_1,z_2)\in g^{-1}K_\delta$, and hence $\operatorname{Pair}(Z) = GK_\delta$. Therefore, the action $G\curvearrowright\operatorname{Pair}(Z)$ is cocompact.
\end{proof}

The following lemma provides a criterion for obtaining a quasisymmetric limit of a sequence of uniformly quasisymmetric maps whose domains converge to the ambient compact space in the Hausdorff sense. Note that, in general, a uniform limit of quasisymmetric maps can be constant. For example, on the compact interval $[0,1]$, the maps
\[
f_k:[0,1]\longrightarrow[0,1],
\qquad
f_k(x)=\frac{x}{k},
\]
are quasisymmetric and converge uniformly to the constant map $f\equiv0$. Thus, the uniform limit of quasisymmetric maps need not be quasisymmetric without a suitable nondegeneracy condition.

\begin{lemma}
\label{lem:tech_1}
    Let $(X,d_X)$ and $(Y,d_Y)$ be compact metric spaces, and let $f_k:D_k\to Y$ be uniformly $\eta$-quasisymmetric, where $D_k\subset X$. Suppose there exist $x_k^1,x_k^2\in D_k$ and $y_k^1,y_k^2\in Y$ such that $f_k(x_k^i)=y_k^i$ and $d_X(x_k^1,x_k^2) \geq \delta$, $d_Y(y_k^1,y_k^2)\ge\delta$, for some $\delta>0$ independent of $k$. If $\operatorname{dist}_H(D_k,X)\to0$, then:
    \begin{enumerate}[$(1)$]
        \item ${f_k}$ has a subsequence converging uniformly to a quasisymmetric map $f\to Y$.

        \item If, in addition, $\operatorname{dist}_H(f_k(D_k),Y)\to0$, then $f$ is a quasisymmetric homeomorphism.
    \end{enumerate}
\end{lemma}

\begin{proof}
    $(1)$ First, we show that $\{f_k:D_k\to Y\}$ is uniformly equicontinuous. Fix $k\in\mathbb N$ and $x,y\in D_k$. Since $d_X(x_k^1,x_k^2)\ge\delta$, there exists $a_k\in\{x_k^1,x_k^2\}$ such that $d_X(x,a_k)\ge\delta/2$. By the $\eta$-quasisymmetry of $f_k$, we have
    \begin{align*}
        & \frac{d_Y(f_k(x),f_k(y))}{d_Y(f_k(x),f_k(a))} \leq \eta\left( \frac{d_X(x,y)}{d_X(x,a)}\right) \leq \eta\left( \frac{2d_X(x,y)}{\delta}\right) \\
        \implies & d_Y(f_k(x),f_k(y)) \leq \mathrm{diam}(Y)\cdot \eta\left( \frac{2d_X(x,y)}{\delta}\right)
    \end{align*} 
    The right-hand side is independent of $k$ and tends to $0$ as $d_X(x,y)\to 0$. Thus, $\{f_k\}$ is uniformly equicontinuous. 

    Since $\operatorname{dist}_H(D_k, X)\to0$ and $Y$ is compact, the Arzel\`a-Ascoli theorem for maps with varying domains yields, after passing to a subsequence, a continuous map $f:X \to Y$ such that $f_k\to f$ uniformly on $D_k$.

    We now show that $f$ is nonconstant. Since $X$ and $Y$ are compact, after passing to a further subsequence, we may assume $x_k^i\to x^i\in X$, $y_k^i\to y^i\in Y$, for $i=1,2$. Because $d_X(x_k^1,x_k^2)\geq\delta$ and $d_Y(y_k^1,y_k^2)\geq\delta$, we have $d_X(x^1,x^2)\geq\delta$ and $d_Y(y^1,y^2)\geq\delta$. Moreover, the uniform convergence on $D_k$ gives 
    \[ 
    f(x^i) = \lim_{k\to\infty}f_k(x_k^i) = \lim_{k\to\infty}y_k^i = y^i. 
    \] 
    Hence $f(x^1)\neq f(x^2)$, and therefore $f$ is nonconstant. 

    It remains to prove that $f$ is quasisymmetric. Let $x,a,b\in X$ be distinct. Since $\operatorname{dist}_H(D_k,X)\to0$, choose $x_k,a_k,b_k\in D_k$ such that $x_k\to x$, $a_k\to a$ and $b_k\to b$. By quasisymmetry,
    \[
    \frac{d_Y(f_k(x_k),f_k(a_k))}      {d_Y(f_k(x_k),f_k(b_k))} \leq \eta\left( \frac{d_X(x_k,a_k)}      {d_X(x_k,b_k)} \right). 
    \] 
    Since $x\neq b$,  
    \[ 
    \frac{d_X(x_k,a_k)} {d_X(x_k,b_k)} \longrightarrow \frac{d_X(x,a)} {d_X(x,b)},
    \]
    while uniform convergence gives
    \[
    f_k(x_k)\to f(x),\qquad f_k(a_k)\to f(a),\qquad f_k(b_k)\to f(b).
    \]
    We claim that $f(x)\neq f(b)$. Otherwise, since $f$ is nonconstant, choose $c\in X\setminus\{x,b\}$ such that $f(c)\neq f(x)$. Applying quasisymmetry to the triple $(x,c,b)$ and passing to the limit gives 
    \[
    d_Y(f(x),f(c)) \leq d_Y(f(x),f(b)) \cdot \eta\!\left(\frac{d_X(x,c)}{d_X(x,b)}\right),
    \]
    which is impossible since the left-hand side is positive and the right-hand side is zero. Hence $f(x)\neq f(b)$. 
    
    We may therefore pass to the limit and obtain
    \[
    \frac{d_Y(f(x),f(a))} {d_Y(f(x),f(b))}\leq\eta\left( \frac{d_X(x,a)} {d_X(x,b)} \right).
    \]
    Thus $f$ is quasisymmetric.

    \noindent
    $(2)$ By $(1)$, after passing to a subsequence, we may assume that
    \[
    \operatorname{dist}(f_k,f|_{D_k})\to0.
    \]
    Set $D_k'=f_k(D_k)$ and let $g_k=f_k^{-1}:D_k'\to X$. Since $\{f_k\}$ is uniformly quasisymmetric, so is $\{g_k\}$. By the additional assumption, $\operatorname{dist}_H(D_k',Y)\to0$, and hence $(1)$ applied to $\{g_k\}$ gives, after passing to a further subsequence, a quasisymmetric map $g:Y\to X$ such that 
    \[
    \operatorname{dist}(g_k,g|_{D_k'})\to0.
    \]
    Since $g_k \circ f_k = \mathrm{id}_{D_k}$ and $f_k \circ g_k = \mathrm{id}_{D_k'}$, the uniform convergence of the sequences $\{f_k\}$ and $\{g_k\}$ implies
    \[
    g \circ f = \mathrm{id}_X \quad \text{and} \quad f \circ g = \mathrm{id}_Y.
    \]
    Hence $f$ is a bijection and therefore a quasisymmetric homeomorphism.
\end{proof} 


\begin{lemma}
\label{lem:tech_2}
    Let $(Z,d_Z)$ be a uniformly perfect compact metric space, \(\delta > 0\) be a constant and \(\eta:(0,\infty) \to (0,\infty)\) be a  strictly increasing homeomorphism. For each $k \in \mathbb{N}$, suppose that we are given 

    \begin{itemize}
        \item a ball \(B_k = B(p_k, R_k) \subseteq Z\),

        \item distinct points $x_k^1, x_k^2 \in B(p_k, \lambda_k R_k)$ with \(d_Z(x_k^1, x_k^2) > \delta_k R_k\), where $\lambda_k, \delta_k > 0$, and $\lambda_k\to 0$, $\delta_k \to 0$,

        \item an $\eta$-quasisymmetric homeomorphism $g_k : Z \to Z$ such that writing $y_k^i := g_k(x_k^i)$ for $i=1,2$, we have 
        \(
        d_Z(y_k^1, y_k^2) > \delta
        \)
    \end{itemize}
    Then the following statements holds:
    \begin{enumerate}[$(1)$]
        \item If $D_k \subseteq B_k$ is $\varepsilon_k R_k$-dense subset in $B_k$, where $\varepsilon_k > 0$, and $ \frac{\varepsilon_k}{\delta_k} \to 0$, then 
        \[\operatorname{dist_H}(g_k(D_k), Z) \to 0 \quad \text{as } k \to \infty.
        \]

        \item The image of $B_k$ becomes asymptotically dense in $Z$, in the sense that
        \[
        \operatorname{diam}(Z \setminus g_k(B_k)) \to 0 \quad \text{as } k \to \infty.
        \]
    \end{enumerate}
\end{lemma}

\begin{proof}
    $(1)$ Since $g_k(D_k)\subset Z$, it suffices to show that $Z\subset N_{r_k}(g_k(D_k))$ for some $r_k\to0$. Fix $z\in Z$ and let $z_k=g_k^{-1}(z)$.

     \noindent\textbf{Case~1.} Suppose \(z_k \in B_k= B(p_k,R_k)\). 
    Since \(D_k\) is \(\epsilon_kR_k\)-dense in \(B_k\), choose $y_k\in D_k$ such that $d_Z(y_k,z_k) \leq \epsilon_k R_k$. Since $d_Z(x_k^1,x_k^2) \geq \delta_k R_k$, there exists \(a_k \in \{x_k^1,x_k^2\}\) satisfying $d_Z(y_k,a_k) \geq \tfrac{\delta_k R_k}{2}$. By quasisymmetry of \(g_k\), 
    \begin{align*}
        & \frac{d_Z \left(g_k(y_k),g_k(z_k)\right)}{d_Z \left(g_k(y_k),g_k(a_k)\right)} \leq \eta \left( \frac{d_Z(y_k,z_k)}{d_Z(y_k,a_k)}\right) 
        \leq \eta\left( \frac{\epsilon_k R_k}{\delta_k R_k/2}\right) \\
        \implies & d_Z \left(g_k(y_k),z\right) \leq \mathrm{diam}(Z) \cdot \eta\left( \frac{2\epsilon_k}{\delta_k}\right)
    \end{align*} 

     \noindent\textbf{Case~2.} Suppose \(z_k \notin B_k= B(p_k,R_k)\). 
    By uniformly perfectness, there exists \(C>1\) such that for all sufficiently large $k$ there exists  points \(w_k \in B(p_k,R_k)\setminus B(p_k,\tfrac{R_k}{C})\). Since \(D_k\) is \(\epsilon_kR_k\)-dense in \(B_k\), there exists \(y_k \in D_k\cap B_k\) such that $d_Z(y_k,w_k) \leq \epsilon_k R_k$. Thus 
    \[
    d_Z(y_k,p_k) \geq d_Z(w_k,p_k) - d_Z(y_k,w_k)\geq \left( \tfrac{1}{C} - \epsilon_k \right) R_k.
    \]
    Choose \(\epsilon_k\) so small so that \(\left( \tfrac{1}{C} - \epsilon_k \right) > C_0\), for some constant \(C_0>0\) independent on \(k\). Hence, $d_Z(y_k,p_k)> C_0 R_k$. Since the maps $g_k$ are uniformly \(\eta\)-quasisymmetric, they are uniformly quasi-M\"{o}bius for some increasing function \(\eta':(0,\infty) \to (0,\infty)\). Hence, 
    \begin{align*}
         \frac{d_Z \left(g_k(z_k),g_k(y_k)\right) \, d_Z \left(g_k(x_k^1),g_k(x_k^2)\right)}{d_Z \left(g_k(z_k),g_k(x_k^1)\right) \, d_Z \left(g_k(y_k),g_k(x_k^2)\right)} 
         & \leq \eta '\left( \frac{d_Z(z_k,y_k) \, d_Z(x_k^1,x_k^2)}{d_Z(z_k,x_k^1) \, d_Z(y_k,x_k^2)}\right) \\
        & \leq \eta'\left( \frac{2 d_Z(z_k,p_k) \cdot 2 \lambda_kR_k}{(d_Z(z_k,p_k) - \lambda_k R_k)(d_Z(y_k,p_k) - \lambda_k R_k)}\right) \\
        & \leq \eta' \left( \frac{4 d_Z(z_k,p_k) \lambda_k}{(d_Z(z_k,p_k) - \lambda_k R_k)(C_0 - \lambda_k)}\right) 
    \end{align*} 
    Since \(C_0\) is independent of \(k\) and \(\lambda_k \to 0\), for sufficiently large $k$, $(C_0-\lambda_k)> \tfrac{C_0}{2}$. Moreover, \(d_Z(z_k,p_k) \geq R_k\) implies \(\lambda_k d_Z(z_k,p_k) \geq \lambda_k R_k\), and hence \(d_Z(z_k,p_k) - \lambda_k R_k \geq (1 - \lambda_k)\, d_Z(z_k,p_k).\)
    For sufficiently large enough $k$, \((1-\lambda_k)> \tfrac{1}{2}\), and thus $d_Z(z_k,p_k) - \lambda_k R_k \geq \tfrac{d_Z(z_k,p_k)}{2}$. Therefore,
    \[
    d_Z \left(g_k(z_k),g_k(y_k)\right) \leq \frac{\mathrm{diam(Z)^2}}{\delta} \, \eta'\left(\frac{16\lambda_k}{C_0}\right)
    \]

    Since, $\tfrac{\epsilon_k}{\delta_k} \to 0$ and $\lambda_k \to 0$, both case~1, and 2 implies the conclusion of $(1)$.

    \bigskip
    
    \noindent $(2)$ Fix \(k\), and let \(y_k',z_k' \in Z \setminus g_k(B_k)\). Set $y_k=g_k^{-1}(y_k')$ and $z_k=g_k^{-1}(z_k')$. Then $y_k,z_k\notin B_k$. Without loss of generality, assume $d_Z(y_k,p_k)\le d_Z(z_k,p_k)$. As in Case~2 of $(1)$, quasi-M\"obius invariance gives
    \[
    \frac{d_Z(z_k',y_k')\,d_Z(g_k(x_k^1),g_k(x_k^2))} {d_Z(z_k',g_k(x_k^1))\,d_Z(y_k',g_k(x_k^2))} \leq \eta'\!\left( \frac{4d_Z(z_k,p_k)\lambda_k} {(d_Z(z_k,p_k)-\lambda_kR_k)(1-\lambda_k)} \right).
    \] 
    For sufficiently large $k$, $d_Z(z_k,p_k)-\lambda_kR_k\ge\tfrac12d_Z(z_k,p_k)$, and $(1-\lambda_k)\ge\tfrac12$. Hence
    \[
    d_Z(y_k',z_k') \leq \frac{\operatorname{diam}(Z)^2}{\delta}\, \eta'(16\lambda_k)\longrightarrow0.
    \]
    Since $y_k',z_k'$ were arbitrary, we have $\operatorname{diam}\bigl(Z\setminus g_k(B_k)\bigr)\to 0$.
\end{proof}

The following lemma shows that cocompactness on pairs, together with the absence of a global fixed point, forces the action to be minimal. 

\begin{lemma}
\label{lem:minimal_action}
    Let $(Z,d_Z)$ be a uniformly perfect compact metric space. Suppose that $G\curvearrowright Z$ is a fixed point free, uniformly quasisymmetric action whose induced action on $\operatorname{Pair}(Z)$ is cocompact. Then the action is minimal; that is, for all $z,z'\in Z$ and $\varepsilon>0$, there exists $g\in G$ such that $d_Z\bigl(g(z'),z\bigr)<\varepsilon$.
\end{lemma}

\begin{proof}
    Fix $z,z'\in Z$ and $\varepsilon>0$. Choose $\lambda_k,R_k>0$ with $\lambda_k,R_k\to0$, and set $B_k=B(z,R_k)$. Let $x_k^1=z$. By uniform perfectness, there exists $C>1$ such that we can choose $x_k^2\in B(z,\lambda_kR_k)\setminus B\left(z,\frac{\lambda_kR_k}{C}\right)$. Then 
    \[
    \frac{\lambda_k}{C}R_k \leq d_Z(x_k^1,x_k^2) < \lambda_kR_k.
    \] 
    Thus we may take $\delta_k=\frac{\lambda_k}{C}$. Since the induced action of $G$ on $\operatorname{Pair}(Z)$ is cocompact, by Lemma~\ref{lem:cocompact_uniformly_separated}, there exists $\delta$ independent of $k$ and $g_k\in G$ such that $d_Z\bigl(g_k(x_k^1),g_k(x_k^2)\bigr) \geq \delta$. Therefore, all the hypotheses of Lemma~\ref{lem:tech_2}~(2) are satisfied. It follows that 
    \[
    \operatorname{diam}\bigl(Z\setminus g_k(B_k)\bigr) \longrightarrow0.
    \]
    Since the action is fixed point free, choose $h\in G$ with $h(z')\ne z'$. For all sufficiently large $k$, $\operatorname{diam}\bigl(Z\setminus g_k(B_k)\bigr) < d_Z\bigl(z',h(z')\bigr)$. Thus $z'$ and $h(z')$ cannot both lie outside $g_k(B_k)$. Hence, either $g_k^{-1}(z')$ or $g_k^{-1}\bigl(h(z')\bigr)$ lies in $B_k$. 
    Since $R_k\to0$, for large $k$, $B_k\subset B(z,\varepsilon)$, and therefore
    \[
    d_Z(g_k^{-1}(z'),z)<\varepsilon \quad\text{or}\quad d_Z(g_k^{-1}h(z'),z)<\varepsilon.
    \]
    In the second case, set $g=g_k^{-1}h$. Thus in either case there exists $g\in G$ with $d_Z(g(z'),z)<\varepsilon$. Hence $G\curvearrowright Z$ is minimal.
\end{proof}

\section{Weak Tangents and Their One-Point Compactifications} 

In this section, we recall weak tangents and their basic properties. We then discuss their one-point compactifications and quasi-M\"obius properties. Finally, under the uniformly quasisymmetric and cocompact action, we show that the one-point compactification of every weak tangent is quasisymmetrically equivalent to $Z$.


\begin{definition}[\cite{BK2002QM}] \label{defn:pointed_conv}
    A \emph{pointed metric space} is a pair $(Z,p)$, where $Z$ is a metric space with metric $d_Z$ and $p\in Z$.    
    A sequence $(Z_k,p_k)$ of pointed metric spaces is said to \emph{converge} to a pointed metric space $(S,p)$ if, for every $R>0$ and every $\varepsilon>0$, there exist $N\in\mathbb{N}$, a subset $M\subseteq B_S(p,R)$, subsets $M_k\subseteq B_{Z_k}(p_k,R)$, and bijections $f_k : M_k \to M$ such that for all $k \geq N$:
    \begin{enumerate}[$(i)$]
        \item $p \in M$, $p_k \in M_k$ and $f_k(p_k) = p$,
        
        \item $M$ and $M_k$ are $\varepsilon$-dense in $B_S(p,R)$ and $B_{Z_k}(p_k,R)$, respectively,
        
        \item for every $x,y \in M_k$, we have $\left| d_{Z_k}(x,y) - d_S \bigl(f_k(x), f_k(y)\bigr) \right| < \varepsilon$.
    \end{enumerate}
\end{definition}

\begin{definition}[Weak tangent] 
\label{defn_weak_tangent}
    A complete (pointed) metric space $(S,p)$ is called a \emph{weak tangent} of a metric space $Z$ if there exist points $p_k\in Z$ and numbers $\alpha_k>0$ with $\alpha_k\to\infty$ such that the sequence of pointed spaces $(\alpha_k Z, p_k)$ converges to $(S,p)$. 

    Here, for $\alpha > 0$, we denote by $\alpha Z$ the metric space $(Z, \alpha d_Z)$. Thus, $\alpha Z$ has the same underlying set as $Z$, with its metric rescaled by the factor $\alpha$. We denote the collection of all weak tangents of $Z$ by
$\mathrm{WT}(Z)$.
\end{definition}


Suppose $(Z,d_Z)$ is a compact metric space and $S\in\mathrm{WT}(Z)$ is a weak tangent of $Z$. If $Z$ is uniformly perfect, then $S$ is unbounded; if $Z$ is doubling, then $S$ is proper; and if $Z$ is Ahlfors $Q$-regular, then so is $S$ \cite{DS1997}. 

Let $(S,d_S)$ be an unbounded locally compact metric space. To compactify $S$, fix $p\in S$ and let $\hat S=S\cup\{\infty\}$ be its one-point compactification. Define
\[
h_p(x) :=
\begin{cases}
    \dfrac{1}{1 + d_S(x,p)} & \text{for } x \in S, \\[6pt]
    0 & \text{for } x = \infty.
\end{cases}
\]
and set
\[
\rho_p(x,y)=h_p(x)h_p(y)d_S(x,y),\quad
\rho_p(x,\infty)=\rho_p(\infty,x)=h_p(x),\quad
\rho_p(\infty,\infty)=0,
\]
for $x,y \in S$. The function $\rho_p$ need not be a metric, since the triangle inequality
may fail. We therefore consider its associated path metric
\[
\hat d_p(x,y)= \inf\left\{ \sum_{i=0}^{k-1}\rho_p(x_i,x_{i+1}): x=x_0,\ldots,x_k=y\ \text{ are in } \hat{S} \right\},
\]
for $x,y \in \hat{S}$.

\begin{lemma}[{\cite[Lemma~2.2] {BK2002QM}}]
\label{lem_metric_one_point_compactification}
    Let $(S,d_S)$ be a metric space, $p\in S$, and $\hat S=S\cup\{\infty\}$ its one-point compactification. Let $\rho_p$ and $\hat d_p$ be as above. Then:
    \begin{enumerate}[$(1)$]
    \item For every $x,y\in\hat{S}$,
    \[
    \frac{1}{4}\,\rho_p(x,y)
    \le
    \hat d_p(x,y)
    \le
    \rho_p(x,y).
    \]

    \item $\hat d_p$ is a metric on $\hat{S}$ inducing the original topology on $\hat{S}$.

    \item The identity map $\operatorname{id}_S:(S,d_S)\to (S,\hat d_p|_S)$ 
    is an $\eta$-quasi-M\"{o}bius homeomorphism, where $\eta(t)=16t$.
    \end{enumerate}
\end{lemma}

\begin{proposition}
\label{prop:sapce_and_their_WTS}
    Suppose $(Z,d_Z)$ is a compact, uniformly perfect and doubling metric space. Let $G \curvearrowright Z$ be a uniformly quasisymmetric action whose induced action $G \curvearrowright \mathrm{Pair}(Z)$ is cocompact. Let $(S,p) \in \mathrm{WT}(Z)$ be a weak tangent of $Z$. Then there exists a quasisymmetric homeomorphism
    \[
    h : (\widehat{S}, \widehat{d}_p) \to Z.
    \]
    Moreover, the restriction
    \[
    h|_S : S \to Z \setminus \{h(\infty)\}
    \]
    is a quasi-M\"{o}bius homeomorphism. 
\end{proposition} 

\begin{proof}
    Let $(S,p)\in\mathrm{WT}(Z)$. By Definition~\ref{defn_weak_tangent}, there exist points $p_k\in Z$ and numbers $\alpha_k>0$ with $\alpha_k\to\infty$ such that $(\alpha_kZ,p_k)\longrightarrow(S,p)$. 
    By pointed convergence, after passing to a subsequence, there exist subsets $\widetilde M_k\subseteq B_S(p,k)$, $M_k\subseteq B_{\alpha_kZ}(p_k,k)$ and bijections $f_k:\widetilde M_k\to M_k$ satisfying
    \begin{enumerate}[$(i)$]
        \item $p\in\widetilde M_k$, $p_k\in M_k$, and $f_k(p)=p_k$;

        \item $\widetilde M_k$ and $M_k$ are $1/k$-dense in $B_S(p,k)$ and $B_{\alpha_kZ}(p_k,k)$, respectively;

        \item for all $x,y\in\widetilde M_k$,
        \begin{equation}\label{eq:pointed_approximation_1}
            \left|d_S(x,y)-d_{\alpha_kZ}\bigl(f_k(x),f_k(y)\bigr)\right| <\frac{1}{k}.
        \end{equation}
    \end{enumerate}

    For each $k$, choose $\widetilde D_k$ to be a maximal $2/k$-separated subset of $\widetilde M_k$ containing $p$, and set $D_k=f_k(\widetilde D_k)$. By maximality, $\widetilde D_k$ is $2/k$-dense in $\widetilde M_k$, and hence $3/k$-dense in $B_S(p,k)$. We claim that $D_k$ is $4/k$-dense in $B_{\alpha_kZ}(p_k,k)$. Indeed, given $z'\in B_{\alpha_kZ}(p_k,k)$, choose $y'\in M_k$ with $d_{\alpha_kZ}(z',y')\leq 1/k$. Write $y=f_k^{-1}(y')\in\widetilde M_k$, and choose $x\in\widetilde D_k$ such that $d_S(x,y)\leq 2/k$. Then, by \eqref{eq:pointed_approximation_1},
    \[
    d_{\alpha_kZ}(f_k(x),y') \leq d_S(x,y)+\frac1k \leq \frac3k.
    \]
    Consequently,
    \[
    d_{\alpha_kZ}(f_k(x),z') \leq d_{\alpha_kZ}(f_k(x),y') +d_{\alpha_kZ}(y',z') \leq \frac4k.
    \]
    Thus $D_k$ is $4/k$-dense in $B_{\alpha_kZ}(p_k,k)$.

    Furthermore, since $\widetilde D_k$ is $2/k$-separated, for distinct $a,b\in\widetilde D_k$ we have $d_S(a,b)\ge2/k$. Hence, by \eqref{eq:pointed_approximation_1},
    \[
    \frac12d_S(a,b) \le d_{\alpha_kZ}(f_k(a),f_k(b)) \le2d_S(a,b).
    \]
    Finally, we may assume that for every $k\ge1$ there are subsets $\widetilde D_k\subset B_S(p,k)$, $D_k\subset B_{\alpha_kZ}(p_k,k)$ and bijections $f_k:\widetilde D_k\to D_k$ such that
    \begin{enumerate}[$(1)$]
        \item $p \in \widetilde D_k$, $p_k \in D_k$, and $f_k(p)=p_k$;
        \item $\widetilde D_k$ and $D_k$ are $4/k$-dense in $B_S(p,k)$ and $B_{\alpha_kZ}(p_k,k)$, respectively;
        \item for all $a,b\in\widetilde D_k$,
        \begin{equation}\label{eq:pointed_approximation_2}
            \frac12d_S(x,y)
            \le d_{\alpha_kZ}(f_k(x),f_k(y))
            \le2d_S(x,y).
        \end{equation}
    \end{enumerate} 

    Since $Z$ is Ahlfors regular so is doubling and uniformly perfect, $S$ is unbounded, proper and uniformly perfect. 
    Choose $q\in S$ with $p\neq q$. We may assume, without loss of generality, that the maximal $2/k$-separated sets $\widetilde D_k$ are chosen so that $q\in \widetilde D_k$ for all $k$; indeed, one can start with the finite set $\{p,q\}$ and extend it to a maximal $2/k$-separated set (for $k$ large enough that $d_S(p,q)>2/k$), possibly increasing the density constant from $4/k$ to $5/k$, which does not affect the subsequent limiting argument. Put $q_k=f_k(q)\in D_k$. Since $B_{\alpha_kZ}(p_k,k) = B_Z(p_k,k/\alpha_k)$, we have $p_k,q_k\in B_Z(p_k,k/\alpha_k)$. Moreover, by the bi-Lipschitz estimate \eqref{eq:pointed_approximation_2},
    \begin{equation}\label{eq:pointed_approximation_3}
    \frac{d_S(p,q)}{2\alpha_k} \leq d_Z(p_k,q_k) \leq \frac{2d_S(p,q)}{\alpha_k}.
    \end{equation}

    Set
    \[
    R_k=\frac{k}{\alpha_k},\qquad  \lambda_k=\frac{2d_S(p,q)}{k},\qquad \delta_k=\frac{d_S(p,q)}{2k}, \qquad \varepsilon_k=\frac{4}{k^2}. 
    \]
    Then $p_k,q_k\in B_Z(p_k,\lambda_kR_k)$ and, by \eqref{eq:pointed_approximation_3}, $d_Z(p_k,q_k)\geq\delta_kR_k$. Moreover, since $D_k$ is $4/k$-dense in $B_{\alpha_kZ}(p_k,k)$, it is $\varepsilon_kR_k = 4/(k\alpha_k)$-dense in $B_Z(p_k,R_k)$.

    Since the action $G\curvearrowright\mathrm{Pair}(Z)$ is cocompact, there exists $\delta>0$ such that, for each $k$, one can choose $g_k\in G$ satisfying 
    \begin{equation}\label{eq:pointed_approximation_4}
        d_Z\bigl(g_k(p_k),g_k(q_k)\bigr)\geq\delta.
    \end{equation}
    Furthermore, $\lambda_k \to 0$, $\delta_k \to 0$ and $\frac{\varepsilon_k}{\delta_k} = \frac{8}{k\,d_S(p,q)} \to 0$, and so Lemma~\ref{lem:tech_2}~$(1)$ gives \begin{equation}\label{eq:pointed_approximation_5}
        \lim_{k\to\infty} \operatorname{dist_H}(g_k(D_k),Z) =0, \quad \text{(with respect to the metric $d_Z$)}.
    \end{equation} 
    
    We now define
    \[
    h_k=g_k\circ f_k: (\widetilde D_k,\hat d_p|_{\widetilde D_k}) \longrightarrow (Z,d_Z).
    \] 
    By Lemma~\ref{lem_metric_one_point_compactification}~$(3)$, the identity map $(S,d_S)\to(S,\hat d_p)$ is uniformly quasi-M\"obius. Together with the uniform bi-Lipschitz bounds in \eqref{eq:pointed_approximation_2} and the uniform quasisymmetry of the action, it follows that $\{h_k\}$ is uniformly quasi-M\"obius. Since $(\hat S,\hat d_p)$ is bounded, the maps are uniformly quasisymmetric. Moreover, \eqref{eq:pointed_approximation_4} implies $d_Z\bigl(h_k(p),h_k(q)\bigr)\geq\delta$. 

    We claim that
    \[
    \operatorname{dist}_H(\widetilde D_k,\hat S)\longrightarrow0.
    \] 

    Indeed, if $x\in B_S(p,k)$, then by density there exists $z_k\in\widetilde D_k$ with $d_S(x,z_k)\le4/k$, and hence
    \[
    \hat d_p(x,z_k)\le\rho_p(x,z_k)\le d_S(x,z_k)\le\frac4k.
    \]
    Thus the points of $B_S(p,k)$ are within $4/k$ of $\widetilde D_k$.     
    It remains to show that points near infinity are also well approximated. Since $S$ is uniformly perfect and unbounded, there exists a constant $C>1$ such that for every sufficiently large $k$ we can choose $y_k\in S$ with ${k}/{C} \le d_S(p,y_k) < k$. Since $y_k\in B_S(p,k)$, the $4/k$-density of $\widetilde D_k$ in $B_S(p,k)$ yields $z_k\in\widetilde D_k$ with $d_S(y_k,z_k)\le4/k$. Consequently,
    \[
    d_S(p,z_k) \ge d_S(p,y_k) - \frac4k \ge \frac{k}{C} - \frac4k \longrightarrow \infty,
    \]
    and therefore
    \[
    \hat d_p(\infty,z_k) \leq \frac{1}{1+d_S(p,z_k)} \leq \frac{1}{1+\frac{k}{C}-\frac4k} \longrightarrow 0.
    \]
    Hence
    \begin{equation}\label{eq:pointed_approximation_6}
        \operatorname{dist}_H(\widetilde D_k,\hat S) = \sup_{x\in\hat S} \operatorname{dist}_{\hat d_p}(x,\widetilde D_k) \longrightarrow0. 
    \end{equation}

    Applying Lemma~\ref{lem:tech_1} to $\{h_k\}$, using \eqref{eq:pointed_approximation_5} and \eqref{eq:pointed_approximation_6}, we obtain, after passing to a subsequence, a quasisymmetric homeomorphism
    \[
    h:(\hat S,\hat d_p)\longrightarrow(Z,d_Z)
    \]
    such that $h_k\to h$ uniformly. 

    Finally, by Lemma~\ref{lem_metric_one_point_compactification}(3), the identity $(S,d_S)\to(S,\hat d_p|_S)$ is quasi-M\"obius. Hence the restriction
    \[
    h|_S:(S,d_S)\longrightarrow Z\setminus\{h(\infty)\}
    \]
    is quasi-M\"{o}bius. 
\end{proof}

\section{Thick Paths: Existence and Density} 

In this section, we study thick paths and their relation to positive modulus and the Ahlfors regular conformal dimension. We first recall the basic properties of thick paths and then prove their existence under certain assumptions on $Z$. Finally, using the uniformly quasisymmetric action, we show that pairs of points that can be joined by thick paths form a dense subset of $Z\times Z$.

\subsection{Thick paths}
\leavevmode

Let $(Z,d_Z,\mu)$ be a metric measure space and $Q\geq 1$. The family of all paths in $Z$ is denoted by $\mathcal{P}$, and the family of constant paths in $Z$ by $\mathcal{C}$.

\begin{definition}[Thick path]\label{defn:thick_path}
    A path $\gamma \in \mathcal{P}$ is called \emph{thick} if for every $\varepsilon > 0$, the family of nonconstant paths in the ball $B(\gamma,\varepsilon) \subset \mathcal{P}$ has positive $Q$-modulus, that is
    \[
    \operatorname{Mod}_Q\bigl(B(\gamma,\varepsilon) \setminus \mathcal{C}\bigr) > 0
    \quad \text{for all } \varepsilon > 0.
    \]
    We denote by $\mathcal{P}_T \subset \mathcal{P}$ the set of thick paths.
\end{definition}


\begin{lemma}[Quasi-M\"{o}bius invariance, {\cite[Lemma~2.10]{BK2005}}] \label{lem:property_thick_path}
    If $(Z,d_Z,\mu)$ is locally compact and Ahlfors $Q$-regular with $Q \geq 1$, then the image of a thick path under any quasi-M\"{o}bius homeomorphism $Z \to Z$ is again thick.
\end{lemma}

 The next lemma shows that, up to a family of paths of $Q$-modulus zero, every nonconstant path is thick. 

\begin{lemma} \label{lem_thik_path_mesures_the_modulus}
    Let $\mathcal{P}$ denote the family of all paths, and let $\mathcal{P}_T$ denote the family of thick paths in $\mathcal{P}$ in $Z$. Then the family of nonconstant paths in $\mathcal{P}$ that are not thick has $Q$-modulus zero. In particular, for any family $\Gamma \subset \mathcal{P}$ of nonconstant paths, one has
    \[
    \operatorname{Mod}_Q(\Gamma \cap \mathcal{P}_T)
    = \operatorname{Mod}_Q(\Gamma).
    \]
\end{lemma}

\subsection{Existence of Thick Path}
\leavevmode

 Keith and Laakso proved that if an Ahlfors regular space attains its Ahlfors regular conformal dimension, then a weak tangent of the space supports a nonconstant curve family of positive modulus. 

\begin{theorem}[{\cite[Corollary~1.0.2]{KL2004}}]
\label{thm:Laakso}
    Let $(Z,d_Z)$ be a complete Ahlfors $Q$-regular metric space, where $Q>1$ is the Ahlfors regular conformal dimension of $Z$. Then there exists a weak tangent $S$ of $Z$ which carries a family $\Gamma$ of nonconstant curves such that $\operatorname{Mod}_Q(\Gamma)>0$.
\end{theorem}

The positive-modulus curve family provided by the Keith--Laakso theorem lives in a weak tangent space rather than in the original space. The following proposition allows us to transfer it to the original space. 

\begin{proposition}
\label{prop:existence_of_thick_path}
    Let $(Z,d_Z)$ be a compact Ahlfors $Q$-regular metric space, where $Q>1$ is the Ahlfors regular conformal dimension of $Z$. Suppose that $Z$ admits a uniformly quasisymmetric action $G\curvearrowright Z$ which is cocompact on the space of distinct pairs $\operatorname{Pair}(Z)$. Then $Z$ carries a family of nonconstant curves of positive $Q$-modulus, or equivalently, contains a thick path. 
\end{proposition}

\begin{proof}
    By Theorem~\ref{thm:Laakso}, there exists a weak tangent $(S,p)\in WT(Z)$ and a family $\Gamma$ of nonconstant curves in $S$ such that $\operatorname{Mod}_Q(\Gamma)>0$. By Lemma~\ref{lem_thik_path_mesures_the_modulus}, the family $\Gamma$ contains a thick path; denoted it by $\gamma$. By Proposition~\ref{prop:sapce_and_their_WTS}, there exists a quasi-M\"{o}bius embedding $h:S \to Z$. Since $\gamma$ is a thick path in $S$, Lemma~\ref{lem:property_thick_path}~(3) implies that $h(\gamma)$ is a thick path in $Z$. Hence $Z$ contains a thick path. 

    Finally, by Lemma~\ref{lem_thik_path_mesures_the_modulus}, the existence of a thick path is equivalent, in the present setting, to the existence of a family of nonconstant curves of positive Q-modulus. This completes the proof. 
\end{proof}

\subsection{Density of Endpoints of Thick Paths}
\leavevmode

The following result shows that, under the hypotheses considered here, thick paths are abundant i.e. pairs of points that can be joined by a thick path form a dense subset of $Z \times Z$. 
To show this fact one of the key ingredient is the following lemma of Bonk and Kleiner, which provides a density point for the initial endpoints of a suitable family of thick paths. 

\begin{lemma}[{\cite[Lemma~4.1]{BK2005}}]\label{lem_density_of_entry_points}
    Suppose $(Z,d_Z)$ is a compact Ahlfors $Q$-regular metric space with $Q > 1$, and that $Z$ carries a family of nonconstant paths of positive $Q$-modulus. Equivalently, $Z$ contains a non-constant thick path. 
    
    Then there exist disjoint open balls $B, B' \subset Z$ such that the set of initial points of thick paths joining $B$ to $B'$ has a point of density in $B$. 
\end{lemma} 

We can now prove the desired density statement. 

\begin{proposition}
\label{prop:end_points_of_thick_paths_are_dense}
    Let $(Z,d_Z)$ be a compact Ahlfors $Q$-regular metric space whose Ahlfors regular conformal dimension is $Q>1$. Suppose that $Z$ admits a fixed point free uniformly quasisymmetric action $G\curvearrowright Z$ such that the induced action on the space of distinct pairs $\operatorname{Pair}(Z)$ is cocompact. Let
    \[
    M := \left\{ (x,y)\in Z\times Z : \text{$x$ and $y$ can be joined by a thick path} \right\}. 
    \] 
    Then $M$ is dense in $Z\times Z$.
\end{proposition} 

\begin{proof}
    By Lemma~\ref{lem_density_of_entry_points}, there exist disjoint open balls $B, B' \subset Z$ and a point of density $x \in B$ of the set of initial points of a family $\Gamma$ of thick paths joining $B$ to $B'$.

    Choose a sequence $\{R_k\}$ with $R_k > 0$ and $R_k \to 0$. Set $B_k := B(x,R_k)$. Let $D_k$ denote the set of initial points of the paths contained in $\Gamma$ that start in $B_k$. Since $x$ is a point of density, we have
    \[
    \varepsilon_k := \frac{\operatorname{dist}_H(D_k,B_k)}{R_k} \to 0 \quad \text{as } k \to \infty.
    \]
    Thus, $D_k$ is $\epsilon_kR_k$-dense in $B_k$. Set $\delta_k := \sqrt{\varepsilon_k}$ and $x_k^1 := x$. Since $Z$ is uniformly perfect, for all sufficiently large $k$ we can choose points $x_k^2 \in Z$ such that
    \[
    \delta_k R_k \le d(x_k^1, x_k^2) \le \lambda_k R_k
    \]
    where $\lambda_k = C \delta_k$ for some $C \ge 1$ independent of $k$.

    By the cocompactness of the action $G \curvearrowright \mathrm{Pair}(Z)$, there exists $\delta>0$ independent of $k$ such that for each $k$ we have an elements $g_k \in G$ satisfying $d_Z(g_k(x_k^1), g_k(x_k^2)) > \delta$. Since $\tfrac{\epsilon_k}{\delta_k}= \sqrt{\epsilon_k} \to 0$, by applying Lemma~\ref{lem:tech_2}, we obtain
    \[
    \operatorname{dist}_H(g_k(D_k), Z) \to 0 \quad \text{and} \quad \operatorname{diam}(Z \setminus g_k(B_k)) \to 0 \quad \text{as } k \to \infty.
    \]

    Passing to a subsequence if necessary, we may assume that the sets $Z \setminus g_k(B_k)$ Hausdorff converge to a singleton set $\{z\}$, for some $z \in Z$. Since $B'$ and $B_k$ are disjoint for all sufficiently large $k$, we have $B' \subset Z \setminus B_k$. Consequently, 
    \[
    \operatorname{dist}_H\bigl(g_k(B'), \{z\}\bigr) \to 0 \quad \text{as } k \to \infty.
    \]

    Now let $z_1, z_2 \in Z$ and $\varepsilon > 0$ be arbitrary. By minimality of the action $G\curvearrowright Z$, see Lemma~\ref{lem:minimal_action}, there exists $g \in G$ such that $g(z) \in B(z_1,\varepsilon)$. For sufficiently large $k$, the Hausdorff convergence above gives 
    \[
    g \circ g_k(B') \subset B(z_1,\varepsilon),
    \]
    while the convergence $\operatorname{dist}_H\bigl( g_k(D_k),Z \bigr) \to 0$ 
    \[
    g \circ g_k(D_k) \cap B(z_2,\varepsilon) \ne \emptyset.
    \]
    
    Every point of $D_k$ is the initial point of a thick path joining $B_k$ to $B'$. Hence, after applying the quasisymmetric (or equivalently, quasi-M\"{o}bius) homeomorphism $g\circ g_k$, each corresponding image path is still thick, by Lemma~\ref{lem:property_thick_path}~(3). Therefore, there is a thick path joining a point of $B(z_1,\varepsilon)$ to a point of $B(z_2,\varepsilon)$. 

    Since $z_1,z_2\in Z$ and $\varepsilon>0$ were arbitrary, it follows that $M$ is dense in $Z\times Z$.
\end{proof} 


\section{Loewner Property of Bowditch Boundary} 

In this section, we prove that the Bowditch boundary in the setting of Theorem~\ref{thm_main_1} is a $Q$-Loewner space. By the criterion of Bonk-Kleiner Proposition~\ref{prop_Loewner_condition_with_ball}, it suffices to establish a uniform modulus estimate for pairs of balls. We obtain this estimate in the following proposition using the uniformly quasisymmetric action and the assumption that $Q$ is the Ahlfors regular conformal dimension of $Z$.

\begin{proposition}\label{prop_ball_modulus}
    Let $(Z,d_Z)$ be a compact Ahlfors $Q$-regular metric space whose Ahlfors regular conformal dimension is $Q>1$. Suppose that $Z$ admits a fixed point free uniformly quasisymmetric action $G\curvearrowright Z$ such that the induced action on the space of distinct pairs $\operatorname{Pair}(Z)$ is cocompact.  
    
    For every $C > 0$, there exist constants $m=m(C) > 0$ and $L=L(C) > 0$ with the following property: if $B, B' \subset Z$ are balls of radius $R>0$ satisfying \( \operatorname{dist}(B,B') \leq C R,\) then the family of paths joining $B$ to $B'$ and having length at most $L R$ has $Q$-modulus at least $m$. 
\end{proposition} 

\begin{proof}
    Fix $C > 0$. We first claim that there exists $m_0=m_0(C) > 0$ such that, whenever $B, B' \subset Z$ are balls of radius $R>0$ satisfying \(\operatorname{dist}(B,B') \leq C R\), we have  \begin{equation}\label{eq:uniform_ball_modulus} 
        \operatorname{Mod}_Q(\Gamma(B,B'))\ge m_0.
    \end{equation} 

    Suppose, for a contradiction, that no such $m_0$ exists. Then there exist balls $B_k=B(z_k,R_k)$ and $B_k'=B(z_k',R_k)$ with $\operatorname{dist}(B_k,B_k')\le CR_k$ and 
    \begin{equation}\label{eq_modulus-vanishing}
        \operatorname{Mod}_Q(\Gamma(B_k,B_k'))\to0.
    \end{equation}

    Since $Z$ is compact, after passing to a subsequence, let $z_k \to z$ and $z_k' \to z'$ for some $z,z' \in Z$.  We must have $R_k \to 0$. Otherwise, after passing to a subsequence, let $R_k \geq r >0$, which implies that for all sufficiently large $k$, 
    \[
    B(z,r/2) \subset B_k \qquad \text{and} \qquad B(z',r/2) \subset B_k'.
    \]
    By the density of endpoints of thick paths, Proposition~\ref{prop:end_points_of_thick_paths_are_dense}, there is a thick path joining these two balls. Hence, for all sufficiently large $k$, we have
    \[
    \operatorname{Mod}_Q\bigl(\Gamma(B_k,B_k')\bigr) \geq \operatorname{Mod}_Q\bigl(\Gamma(B(z,r/2),B(z',r/2))\bigr) >0,
    \]
    contradicting \eqref{eq_modulus-vanishing}. Thus $R_k \to 0$. 

    Since $d_Z(z_k,z_k')\le(C+2)R_k$, cocompactness on pairs gives $g_k\in G$ and $\delta>0$ such that 
    \begin{equation}\label{eq:separation_1}
        d_Z\bigl(g_k(z_k),g_k(z_k')\bigr)\geq\delta.
    \end{equation} 
    If $x_k\notin B_k$, then $d_Z(z_k,x_k)\ge R_k$ and hence
    \[
    \frac{d_Z(z_k,z_k')}{d_Z(z_k,x_k)}\le C+2.
    \]
    Uniform quasisymmetry and the above separation inequality \eqref{eq:separation_1} implies that
    \begin{align*}
        & \frac{d_Z\bigl(g_k(z_k),g_k(z_k')\bigr)} {d_Z\bigl(g_k(z_k),g_k(x_k)\bigr)} \leq \eta\left( \frac{d_Z(z_k,z_k')}{d_Z(z_k,x_k)} \right) \leq \eta(C+2), \\
        \implies & d_Z\bigl(g_k(z_k),g_k(y_k)\bigr) \geq \frac{\delta}{\eta(C+2)}.
    \end{align*}  
    Thus, for $0< \varepsilon < \frac{\delta}{2\eta(C+2)}$ independent of $k$,
    \[
    B(g_k(z_k),2\varepsilon)\subset g_k(B_k).
    \]
    Similarly,
    \[
    B(g_k(z_k'),2\varepsilon)\subset g_k(B_k').
    \]    
    Since $Z$ is compact, Passing to a subsequence, let $g_k(z_k)\to w$ and $g_k(z_k')\to w'$. Hence, for all large $k$,
    \[
    \widetilde B:=B(w,\varepsilon)\subset g_k(B_k), \qquad \widetilde B':=B(w',\varepsilon)\subset g_k(B_k').
    \]

    Since the $g_k$ are uniformly quasisymmetric, Tyson's theorem, Theorem~\ref{thm_Tyson}, gives a constant $K>0$, independent of $k$, such that
    \[
    \operatorname{Mod}_Q(\Gamma(B_k,B_k')) \ge K\,\operatorname{Mod}_Q \bigl(\Gamma(g_k(B_k),g_k(B_k'))\bigr).
    \]
    By monotonicity,
    \[
    \operatorname{Mod}_Q(\Gamma(B_k,B_k'))
    \ge K \,\operatorname{Mod}_Q(\Gamma(\widetilde B,\widetilde B'))>0,
    \]
    where the last inequality follows from the density of endpoints of thick paths. This contradicts the assumption and proves \eqref{eq:uniform_ball_modulus}. Hence, the claim is proved.

    Finally, choose $L\ge2$ sufficiently large as in \cite[Lemma~2.8]{BK2005} so that every family of paths starting in a ball of radius $R$ and having length at least $LR$ has $Q$-modulus at most $m_0/2$. For arbitrary balls $B,B'$ of radius $R$, let $\Gamma_1$ consist of paths joining $B$ to $B'$ of length at most $LR$, and let $\Gamma_2$ consist of the remaining paths. Then by using the above claim
    \[
    m_0\le\operatorname{Mod}_Q(\Gamma(B,B')) \le\operatorname{Mod}_Q(\Gamma_1)+\operatorname{Mod}_Q(\Gamma_2) \le\operatorname{Mod}_Q(\Gamma_1)+\frac{m_0}{2}. 
    \]
    Therefore
    \[
    \operatorname{Mod}_Q(\Gamma_1)\ge \frac{m_0}{2}>0.
    \]
    Taking $m := \frac{m_0}{2}>0$ proves the result.
\end{proof}

To complete the proof of Theorem~\ref{thm_main_1}, we now assume that $Z$ is metric space as in Theorem~\ref{thm_main_1}. Since $Z$ is compact, it is complete. Poposition~\ref{prop_ball_modulus} provides that for every $C>0$, there exists constants $m=m(C)>0$ and $L=L(C)>0$ satisfying the hypothesis of Proposition~\ref{prop_Loewner_condition_with_ball}. Therefore, $Z$ is a $Q$-Loewner space. This proves Theorem~\ref{thm_main_1}.

\bibliographystyle{amsalpha}
\bibliography{Bibliography} 


\end{document}